\documentclass[12pt,reqno]{amsart}
\usepackage{indentfirst, amssymb, amsmath, amsthm, mathrsfs, setspace, indentfirst, enumerate,  mathrsfs, amsmath, amsthm}
\usepackage[bookmarksnumbered, colorlinks, plainpages]{hyperref}
\usepackage{mathrsfs}
\usepackage{cite}
\usepackage{xcolor}
\usepackage{colortbl}
\usepackage{array}
\usepackage{booktabs}
\definecolor{headernavy}{RGB}{50,100,200}
\definecolor{rowtint}{RGB}{237,244,249}
\definecolor{rulecolor}{RGB}{47,84,120}
\usepackage{tikz}
\usepackage{pgfplots}
\usepgfplotslibrary{fillbetween}
\pgfplotsset{compat=1.18}
\usetikzlibrary{arrows.meta}
\usetikzlibrary{backgrounds,calc,positioning}
\usepgfplotslibrary{fillbetween}
\usepackage{float}
\usepackage{tabularx}
\usepgfplotslibrary{fillbetween}
\usepackage[table]{xcolor}
\usepackage[dvipsnames]{xcolor}
\usetikzlibrary{positioning,calc,shadows.blur}

\definecolor{headerblue}{RGB}{220,230,245}
\definecolor{lightblue}{RGB}{248,250,255}

\newtheorem*{theoA}{Theorem A}
\newtheorem*{theoB}{Theorem B}
\newtheorem*{theoC}{Theorem C}
\newtheorem*{theoD}{Theorem D}

\newtheorem*{cor A}{Corollary A}
\newtheorem*{cor B}{Corollary B}

\newtheorem{theo}{Theorem}[section]
\newtheorem{lem}{Lemma}[section]
\newtheorem{cor}{Corollary}[section]
\newtheorem{ques}{Question}[section]

\newtheorem{rem}{Remark}[section]

\newcommand{\be}{\begin{equation}}
\newcommand{\ee}{\end{equation}}
\newcommand{\beas}{\begin{eqnarray*}}
\newcommand{\eeas}{\end{eqnarray*}}
\newcommand{\bea}{\begin{eqnarray}}
\newcommand{\eea}{\end{eqnarray}}

\numberwithin{equation}{section}
\begin{document}
\title[ B\MakeLowercase{ohr-type inequalities for operator-valued analytic.....}]{\Large B\MakeLowercase{ohr-Type Inequalities for Certain Integral Transforms of Operator-Valued Analytic Functions on Shifted Disc}}

\date{}
\author[G. H\MakeLowercase{aldar}, K. A\MakeLowercase{li}, S. M\MakeLowercase{ajumder} \MakeLowercase{and} D. C. P\MakeLowercase{ramanik}]{G\MakeLowercase{outam} H\MakeLowercase{aldar}$^{1*}$, K\MakeLowercase{awsar} A\MakeLowercase{li}$^{2}$, S\MakeLowercase{ujoy} M\MakeLowercase{ajumder}$^{3}$ \MakeLowercase{and} D\MakeLowercase{ilip} C\MakeLowercase{handra} P\MakeLowercase{ramanik}$^{2}$}
\address{$^1$Department of Mathematics, Ghani Khan Choudhury Institute of Engineering and Technology, Narayanpur, Malda-732141, West Bengal, India.}
\email{goutamiitm@gmail.com, goutamiit1986@gmail.com}
\address{$^2$Department of Mathematics, University of North Bengal, Raja Rammohunpur, Darjeeling, West Bengal, India}
\email{ali.math.du@gmail.com, dcpramanik.nbu2012@gmail.com}
\address{$^{3}$Sujoy Majumder, Department of Mathematics, Raiganj University, Raiganj, West Bengal-733134, India.}
\email{sm05math@gmail.com, sjm@raiganjuniversity.ac.in}
\renewcommand{\thefootnote}{}
\footnote{*\emph{Corresponding Author}: }

\renewcommand{\thefootnote}{\arabic{footnote}}
\setcounter{footnote}{0}

\begin{abstract}
In this paper, we investigate Bohr-type inequalities for certain integral transforms of some bounded holomorphic operator-valued functions defined on a shifted disc. More precisely, we consider functions $f\in H^{\infty}(\Omega_\gamma,\mathcal{B}(\mathcal{H}))$, where $\mathcal{B}(\mathcal{H})$ denotes the space of bounded linear operators on a complex Hilbert space $\mathcal{H}$, and study the Bohr phenomena associated with the series representations in \eqref{e1.2a} and \eqref{e2.1}. Moreover, we determine the sharp radii $R_\gamma>0$ for which the corresponding majorant series satisfy the required inequalities for all $|z|\leq R_\gamma$. One of our main results extends some recent result due to Kumar and Sahoo \cite{Kumar-Sahoo-Meditarr-2023} to a large extent, and the other one is new in the literature
\end{abstract}
\thanks{Typeset by \AmS -\LaTeX}
\maketitle

\section{\bf Introduction and Preliminaries}
We denote by $\mathbb{D}:=\{z\in\mathbb{C}:|z|<1\}$ the unit disk in the complex plane, and $H(\mathbb{D})$ denote the class of all analytic functions defined on $\mathbb{D}$. We consider the class $\mathcal{B}:=\{f\in H:|f(z)|\leq 1,\ z\in\mathbb{D}\}$. A fundamental result of Bohr \cite{Bohr-PLMS-1914} asserts that if $f(z)=\sum_{n=0}^{\infty}a_nz^n\in\mathcal{B}$, then
\[
\sum_{n=0}^{\infty}|a_n|r^n\leq 1\qquad \text{for}\;\;r\leq\frac{1}{3},
\]
where the constant $\frac{1}{3}$ is best possible. This celebrated result, known as Bohr's theorem \cite{Bohr-PLMS-1914}, has played a significant role in the development of the Bohr radius problem in geometric function theory. It is worth mentioning that Bohr originally obtained the result for $r\leq\frac{1}{6}$, while the sharp radius $\frac{1}{3}$ was subsequently established through the contributions of Riesz, Schur, and Wiener. The constant $\frac{1}{3}$ is referred to as the \emph{Bohr radius} of the class $\mathcal{B}$. For the subclass $ \mathcal{B}_0:=\{f\in\mathcal{B}: f(0)=0\}$, Bombieri \cite{Bom-Math-1962} proved that the corresponding Bohr radius is $\frac{1}{\sqrt{2}}$. Several extensions and generalizations of this result have subsequently been obtained; see, for instance, \cite{Ponnusamy-Wriths-CMFT-2020} and the references therein.

The study of Bohr radii is not restricted to the class $\mathcal{B}$. In recent years, Bohr-type inequalities and Bohr radii have been investigated for a wide variety of subclasses of analytic and univalent functions, as well as for certain integral operators. These investigations have led to numerous interesting problems and results in geometric function theory. In particular, an interesting application of the Bohr radius problem for the class $\mathcal{B}$ can be found in \cite{Ponnusamy-Vijaykumar-Wriths}.

The Bohr phenomenon has been extensively investigated for various subclasses of analytic functions and in several related settings. Ali et al. \cite{Ali-Barnard-Solynin-2017} initiated the study of the Bohr problem for odd analytic functions, which was subsequently settled by Kayumov and Ponnusamy \cite{Kayum-Ponnusamy-CMFT-2017}. Furthermore, Kayumov and Ponnusamy \cite{Kayum-Ponnusamy-JMMA-2018} extended the Bohr radius problem to a broader class of odd analytic functions. Bhowmik and Das \cite{Bhowmik-Das-JMAA-2018} investigated the Bohr radius for certain families of analytic univalent (one-to-one) functions. The Bohr phenomenon for functions belonging to Hardy spaces was studied in \cite{Khavinson-CMFT-2004}, while the corresponding problem for Bloch functions was considered in \cite{Kayum-Ponnusamy-Math. Nachr.-2018}. Moreover, the Bohr phenomenon for families of functions satisfying quasi-subordination was investigated in \cite{Alk-Kay-Ponnu-2019-Proc.,Ponn-Vijay-Writh-JMMA-2022}. Further developments and related results can be found in \cite{ABu-Ali-JMAA-2011,Abu-Ali-Ng-MathNachr-2017,Abu-Com.Var-2010,Kayum-Ponnusamy-CRM-2018,Kayum-Ponnusamy-Ann.Acad.-2019,Liu-Shao-Xu_JIA-2018,Ben-Dah-Khav-CMFT-2004,Aizenberg-PAMS-1999,Aizenberg-Aytuna-Djakov-JMAA-2001,Aizenberg-StudMat-2007,Aizenberg,Ali-Jain-Ravichandran-2019-Results Math} and the references therein. 

Regarding complex integral operators, Kayumov \textit{et al.} \cite{Kay-Ponnu-CRMath-2020} initiated, in 2020, the study of the Bohr radius of the classical Ces\'aro operator defined on the classes $\mathcal{B}$ and $\mathcal{B}_0$. The Ces\'aro operator is studied in \cite{Hardy-Math. Z.-1932,Siskakis-J.Lond.-1987,Siskakis-Proc.Am.-1990} and defined as 
\bea\label{e1.1} \mathcal{C}f(z)=\int_{0}^{1}\frac{f(tz)}{1-tz}dz=\sum_{n=0}^{\infty}\left(\frac{1}{n+1}\sum_{k-0}^{n}a_k\right)z^n,\eea where $f(z)=\sum_{n=0}^{\infty}a_nz^n$ is analytic in $\mathbb{D}$. Kayumov \textit{et al.} \cite{Kay-Ponnu-CRMath-2020} observed that for each $r\in(0,1)$ and $f\in\mathcal{B}$, the absolute value of the Ces\'aro operator defined in \eqref{e1.1} satisfies the sharp inequality \beas |\mathcal{C}f(z)|\leq \frac{1}{r}\ln\frac{1}{1-r}.\eeas 
\par Regarding the majorant series $\mathcal{C}f(f)$ associated with \eqref{e1.1}, defined by
\beas \mathcal{C}_f(r):=\sum_{n=0}^{\infty}\left(\frac{1}{n+1}\sum_{k-0}^{n}|a_k|\right)r^n,\eeas
Kayumov \textit{et al.} \cite{Kay-Ponnu-CRMath-2020} obtained the following interesting result.
\begin{theoA}\cite{Kay-Ponnu-CRMath-2020}
If $f=\sum_{n-0}^{\infty}a_nz^n\in\mathcal{B}$, then \beas \mathcal{C}_f(r)\leq\frac{1}{r}\ln\frac{1}{1-r},\eeas for $r\leq R$, where $ R=0.5335....$ is the positive root of the equation \beas 2x-3(1-x)\ln(1-x)=0.\eeas Furthermore, the number $R$ is the best possible.	
\end{theoA}
\par Inspired by the above result, Kumar and Sahoo \cite{Kumar-Sahoo-Meditarr-2023} studied the Bohr radius phenomena of $\beta$ Ces\'aro operator defined by \bea\label{e1.2a} \mathcal{T}f(z)=\int_{0}^{1}\frac{f(tz)}{(1-tz)^\beta}dz=\sum_{n=0}^{\infty}\left(\frac{1}{n+1}\sum_{k-0}^{n}\frac{\Gamma(n-k+\beta)}{\Gamma(n-k+1)\Gamma(\beta)}a_k\right)z^n,\eea where $\beta>0$ and $f(z)=\sum_{n=0}^{\infty}a_nz^n$ is analytic in $\mathbb{D}$. They obtained the following interesting result.
\begin{theoB}\cite{Kumar-Sahoo-Meditarr-2023}
Let $f(z)\in\mathcal{B}$ and $0<\beta\neq
1$. Then \bea\label{e1.2} \sum_{n=0}^{\infty}\left(\frac{1}{n+1}\sum_{k-0}^{n}\frac{\Gamma(n-k+\beta)}{\Gamma(n-k+1)\Gamma(\beta)}|a_k|\right)r^n\leq \frac{1-(1-r)^{1-\beta}}{r(1-\beta)},\eea for $r\leq R_\beta$, where $R_\beta$ is the positive root of the equation \beas \frac{3[1-(1-x)^{1-\beta}]}{1-\beta}-\frac{2[(1-x)^{-\beta}-1]}{\beta},\eeas and the radius can not be improved further.
\end{theoB}
\par For $0\leq\gamma< 1$, let us consider the simply connected domain $\Omega_\gamma$ defined by \beas \Omega_\gamma=\left\{z\in\mathbb{C}: \left|z+\frac{\gamma}{1-\gamma}\right|<\frac{1}{1-\gamma}\right\}.\eeas
\par It is easy to observe that $\mathbb{D}\subset\Omega_\gamma$ for each $0\leq \gamma<1$. As a generalization of the Bohr inequality and its applications on $\Omega_\gamma$, Fournier and
Ruscheweyh \cite{Fournier-Ruscheweyh-CRM-2010}, in 2010, proved the following result.
\begin{theoC}\cite{Fournier-Ruscheweyh-CRM-2010}
	If $f\in\mathcal{B}(\Omega_\gamma)$ with $f(z)=\sum_{n=0}^{n}a_nz^n$ in $\mathbb{D}$, then 
	\beas \sum_{n=0}^{\infty}|a_n|r^n\leq1\quad \text{for}\;\; r\leq R_\gamma:=\frac{1+\gamma}{3+\gamma}.\eeas
	Moreover, $\sum_{n=0}^{\infty}|a_n|R_\gamma^n=1$ holds for a function $f(z)=\sum_{n=0}^{\infty}a_nz^n$ in $\mathcal{B}(\Omega_\gamma)$ if and only if $f(z)=c$ with $|c|=1$.
\end{theoC}
\par Some interesting results on $\Omega_\gamma$ can be found in \cite{Ponnusamy-Rasila-Results Mat-2021,Evdoridis-Ponnusamy-Indag. Math.-2019,Kumar-Complex Var-2023}.
\par It is well known that for $f(z)=\sum_{n=0}^{\infty}a_nz^n\in H(\mathbb{D})$, the classical Voltera operator $V$ is defined as \beas (Vf)(z)=\int_{0}^{z} f(t)dt=\sum_{n=0}^{\infty}\frac{a_n}{n+1}z^{n+1},\qquad z\in \mathbb{D},\eeas and corresponding mejorant series is defined by $M_{V_f}(r)=\sum_{n=0}^{\infty}\frac{|a_n|}{n+1}|z|^{n+1}$.
\section{\bf Bohr inequalities on some complex Banach Space}

\par The main objective of this paper is to investigate the vector-valued analogue of Equations \eqref{e1.2} and the majorant series associated with the classical Volterra operator \eqref{e2.1} on simply connected domains in complex Banach spaces. To develop the necessary framework and facilitate the subsequent discussion, we first recall the basic notations and definitions that will be used throughout the paper. 
\par Let $H^{\infty}(\Omega,X)$ denote the space of bounded analytic functions from $\Omega$ into a complex Banach space $X$ equipped with the norm  $||f||_{H^{\infty}(\mathbb{D},X)}=\text{sup}_{z\in\Omega}||f(z)||$, where $\Omega$ is a simply connected domain containing $\mathbb{D}$ in $\mathbb{C}$. As $H^{\infty}(\Omega,X)\subset H^{\infty}(\mathbb{D},X)$, every $f\in H^{\infty}(\Omega,X)$ can be expressed into the Taylor's series on $\mathbb{D}$ as \beas f(z)=\sum_{n=0}^{\infty}A_nz^n,\eeas where $z\in\mathbb{D}$, $A_n\in X$ for $n\in\mathbb{N}\cup\{0\}$.

Throughout this paper, let $\mathcal{B}(\mathcal{H})$ denote the space of all bounded linear operators acting on a complex Hilbert space $\mathcal{H}$. For $T\in\mathcal{B}(\mathcal{H})$, we use $||T||$ to denote the operator norm of $T$. The adjoint operator $T^*\colon\mathcal{H}\to\mathcal{H}$ is uniquely determined by $\langle Tx,y\rangle=\langle x,T^*y\rangle$ for all $x,y\in\mathcal{H}$. An operator $T\in\mathcal{B}(\mathcal{H})$ is called \emph{normal} if $T^*T=TT^*$, \emph{self-adjoint} if $T^*=T$, and \emph{positive} if $\langle Tx,x\rangle\geq 0$ for all $x\in\mathcal{H}$. The absolute value of $T$ is defined by $|T|:=(T^*T)^{1/2}$, where $S^{1/2}$ denotes the unique positive square root of a positive operator $S$. We denote by $I$ the identity operator on $\mathcal{H}$. For more details we refer the reader to \cite{Graham-Kohr-2003}. 

Recently, in \cite{Allu-Halder-Proc.Edin.-2024} (see also \cite{Allu-Halder-Bull. Canad.-2023}), the authors extensively studied the Bohr inequality for the class $f\in H^{\infty}(\Omega_\gamma, \mathcal{B}(\mathcal{H}))$, and obtained the following result.
\begin{theoD}\cite{Allu-Halder-Proc.Edin.-2024}
$ Let f\in H^{\infty}(\Omega_\gamma, \mathcal{B}(\mathcal{H}))$, $0\leq \gamma<1$ be given by $f(z)= \sum_{n=0}^{\infty}A_nz^n$ in $\mathbb{D}$ with $||f||_{H^{\infty}(\Omega_\gamma, \mathcal{B}(\mathcal{H}))}\leq1$ and $A_0=a_0I$ for $|a_0|<1$. Then \beas \sum_{n=0}^{\infty}||A_n||r^n\leq1 \quad \text{for}\;\; r\leq R_\gamma:=\frac{1+\gamma}{3+\gamma}.\eeas
\end{theoD}
\par Let $f\in H^{\infty}(\Omega_\gamma,\mathcal{B}(\mathcal{H}))$ be given by $f(z)=\sum_{n=0}^{\infty}A_nz^n$ in $\mathbb{D}$ with $||f(z)||_{H^{\infty}(\Omega_\gamma,\mathcal{B}(\mathcal{H}))}\leq1$, where $||A_0||<1$. We define the Voltera operator as \bea\label{e2.1} (Vf)(z)=\int_{0}^{z}f(t)dt=\sum_{n=0}^{\infty}\frac{A_n}{n+1}z^{n+1},\eea and the corresponding mejorant series is defined by \beas M_V^f(r):=\sum_{n=0}^{\infty}\frac{||A_n||}{n+1}r^{n+1}.\eeas
\par The above considerations naturally lead to the following questions concerning Bohr-type inequalities for operator-valued analytic functions defined on a shifted disc. 
\begin{ques}
What can be said about the Bohr-type inequalities associated with \eqref{e1.2a} in the shifted disc $\mathcal{B}(\Omega_\gamma)$, for functions $f\in H^{\infty}(\Omega_\gamma,\mathcal{B}(\mathcal{H}))?$ More precisely, can one determine a radius $R_\gamma>0$ such that the corresponding majorant series of \eqref{e1.2a} satisfies the desired Bohr inequality for every $|z|\leq R_\gamma$? Furthermore, can the radius $R_\gamma$ be determined explicitly, and under what conditions is it sharp?
\end{ques}

\begin{ques}
What can be said about the Bohr-type inequalities corresponding to \eqref{e2.1} in the shifted disc $\mathcal{B}(\Omega_\gamma)$, where $f\in H^{\infty}(\Omega_\gamma,\mathcal{B}(\mathcal{H}))?$
In particular, can one obtain an explicit Bohr radius for \eqref{e2.1}, and can the resulting radius be shown to be optimal? 
\end{ques}
\par The main objective of this article is to provide affirmative answers to the above questions. More precisely, we establish Bohr-type inequalities for the functions under consideration in the shifted disc $\Omega_\gamma$ and determine the corresponding Bohr radii. In addition, we investigate the sharpness of these radii by constructing suitable extremal functions. Thus, our results extend the classical Bohr phenomenon to the setting of operator-valued analytic functions on shifted discs.
\section{\bf Main Results}
Regarding the first question, we obatin the following result.
\begin{theo}\label{Ali-Haldar-thm-2.1}
For $0\leq \gamma<1$ and $0<\beta\neq1$, consider a function $f\in H^\infty(\Omega_\gamma,\mathcal{B}(\mathcal{H}))$ represented in $\mathbb{D}$ by the power series $f(z)=\sum_{n=0}^{\infty}A_nz^n,$
and satisfying $||f||_{{H^\infty(\Omega\gamma,\mathcal{B}(\mathcal{H}))}}\leq 1.$ Assume, in addition, that the constant coefficient is of the form $A_0=a_0I$, where $a_0\in\mathbb{C}$ and $|a_0|<1$. Then \beas \sum_{n=0}^{\infty}\left( \frac{1}{n+1}\sum_{k=0}^{n}\frac{\Gamma\left( n-k +\beta\right)}{\Gamma\left(n-k+1\right)\Gamma(\beta)} ||A_k||\right)r^n \leq\frac{1}{r}\left(\frac{1-(1-r)^{1-\beta}}{1-\beta}\right)
\eeas for $r\leq R_{\beta}$, where $R_{\beta}$ is the unique positive root of the equation \beas (3+\gamma)\frac{1-(1-x)^{1-\beta}}{1-\beta}-2\frac{(1-x)^{-\beta}-1}{\beta}=0,\eeas and the radius $R_{\beta}$ is the best possible.
\end{theo}
\noindent {\bf Applications of Theorem \ref{Ali-Haldar-thm-2.1}:}\vspace{1.2mm}
\par As applications of Theorem \ref{Ali-Haldar-thm-2.1}, one can easily deduce the following results of Bohr’s theorem in various settings:
\begin{enumerate}
\item [1.] For $\beta\longrightarrow1$, one can easily see that \beas \sum_{n=0}^{\infty}\left( \frac{1}{n+1}\sum_{k=0}^{n}\frac{\Gamma\left( n-k +\beta\right)}{\Gamma\left(n-k+1\right)\Gamma(\beta)} ||A_k||\right)r^n\longrightarrow\sum_{n=0}^{\infty}\frac{1}{n+1}\sum_{k=0}^{n}||A_k||\eeas and 
\beas \frac{1-(1-r)^{1-\beta}}{1-\beta}\longrightarrow \ln\left(\frac{1}{1-r}\right).\eeas Therefore, Theorem \ref{Ali-Haldar-thm-2.1} yields \beas \sum_{n=0}^{\infty}\frac{1}{n+1}\sum_{k=0}^{n}||A_k|| \leq\frac{1}{r}\ln\left(\frac{1}{1-r}\right)\eeas asymptotically for $r\leq R$, where $R$ is the unique positive root of the equation \beas \boxed{2x=(3+\gamma)(1-x)\ln\frac{1}{1-x}\approx0}.\eeas This is precisely the classical Ces\'aro Bohr-radius theorem of vector valued function on shifted disc.
\item [2.] Let $\mathcal{H}=\mathbb{C},$ $\mathcal{B}(\mathcal{H})=\mathbb{C}$ and $A_n=a_n\in\mathbb{C}$ for $n\in\mathbb{N}$. Then Theorem \ref{Ali-Haldar-thm-2.1} reduces to the scalar valued theorem as follows.
\end{enumerate}
\begin{cor}\label{cor 3.1}
Let $0\leq\gamma<1$, $0<\beta\neq1$ and $f:\Omega_\gamma\to\mathbb{C}$ be analytic on $\Omega_\gamma$ with $|f(z)|\leq1$. Then \beas \sum_{n=0}^{\infty}\left( \frac{1}{n+1}\sum_{k=0}^{n}\frac{\Gamma\left( n-k +\beta\right)}{\Gamma\left(n-k+1\right)\Gamma(\beta)} |a_k|\right)r^n \leq\frac{1}{r}\left(\frac{1-(1-r)^{1-\beta}}{1-\beta}\right)
\eeas for $r\leq R_{\beta}$, where $R_{\beta}$ is the unique positive root of the equation \beas (3+\gamma)\frac{1-(1-x)^{1-\beta}}{1-\beta}-2\frac{(1-x)^{-\beta}-1}{\beta}=0,\eeas and the radius $R_{\beta}$ is the best possible.
\end{cor}
\begin{rem}
For $\gamma=0$, $\Omega_\gamma$ reduces to $\mathbb{D}$. Therefore, Corollary {\em{\ref{cor 3.1}}} becomes Theorem B.	
\end{rem}
\par Regarding the second question, we obtain the next result as follows.
\begin{theo}\label{t4.4}
Let  $f\in H^\infty(\Omega_\gamma,\mathcal{B}(\mathcal{H}))$ be represented by the power series $f(z)=\sum_{n=0}^{\infty}A_nz^n$ in $\mathbb{D}$ with $||f||_{{H^\infty(\Omega\gamma,\mathcal{B}(\mathcal{H}))}}\leq 1$, $0\leq \gamma<1$. In addition, the constant coefficient $A_0=a_0U$, where $a_0\in\mathbb{C}$ and $|a_0|<1$. Then \beas M_V^f(r)\leq r\quad \text{for}\quad r\leq R_{\gamma},\eeas where $R_\gamma$ is the positive root of the equation \beas (3+\gamma)x+2\ln(1-x)=0.\eeas Moreover the Bohr radius $R_\gamma\in(0,1)$ is sharp.
\end{theo}
\begin{rem}
Replacing $\mathcal{H}=\mathbb{C}$, $\mathcal{B}(\mathcal{H})=\mathbb{C}$, $A_n=a_n\in\mathbb{C}$ for $n\in\mathbb{N}$ and $U=1$ in Theorem {\em\ref{t4.4}}, we one can easily obtain following scalar valued Volterra Bohr result on shifted disc. 
\end{rem}
\begin{cor}\label{c1}
	Let $f$ be an analytic function on $\mathbb{D}$ with $|f(z)|\leq1$ with Taylor series expansion $f(z)=\sum_{n=0}^{\infty}a_nz^n$, where $a_n\in\mathbb{C}$. Then the mejorant series $M_V^f(r):\sum_{n=0}^{\infty}\frac{|a_n|}{n+1}r^{n+1}$ of the Voltera operator \beas (Vf)(z)=\int_{0}^{z}f(t)dt=\sum_{n=0}^{\infty}\frac{a_n}{n+1}z^{n+1}\eeas satisfies \beas M_V^f(r)\leq r\quad \text{for}\;\;r\leq R_\gamma,\eeas
	where $R\in(0,1)$ is the root of the equation \beas 3x+2\ln (1-x)=0.\eeas Moreover, the radius $R$ is can not be improved further.
\end{cor}
\begin{rem}
For $\gamma=0$, we have $\Omega_0=\mathbb{D}$ and Theorem \ref{t4.4} becomes the result on unit disc. In this case, we easily see that the Bohr radius is $\boxed{R_0\approx0.5828116439}$.
\end{rem}

\section{\bf Key Lemmas}
Throughout this paper, we shall repeatedly use the following auxiliary lemmas, which provide the essential tools for establishing the main results.
\begin{lem}\label{lem-1}\cite{Allu-Halder-Proc.Edin.-2024}
Let $f:\Omega_{\gamma}\to \mathcal{B}(\mathcal{H})$ be a bounded analytic function with expression 
$f(z)=\sum_{n=1}^{\infty} A_nz^n$ in $\mathbb{D}$ such that $ A_n \in \mathcal{B}(\mathcal{H})$ for all 
$n\in\mathbb{N}_0$ and $A_0$ is normal. Then
\beas ||A_n||\leq \frac{||I-|A_0|^2||}{1+\gamma}\;\;\text{for}\;\; n\geq1.
\eeas
\end{lem}
\begin{lem}\label{l2.2}\cite{Anderson-Rovnyak-mathematika-2006}
Let $B(z)$ be an analytic function with values in $\mathcal{\mathcal{H}}$ and satisfying $\left||B(z)|\right|\leq1$ on $\mathbb{D}$. Then \beas (1-|a|)^{n-1}\left|\left|\frac{B^{(n)}(a)}{n!}\right|\right|\leq \frac{||I-B(a)^{*}B(a)||^{1/2}||I-B(a)B(a)^*||^{1/2}}{1-|a|^2}\eeas for each $a\in\mathbb{D}$ and $n\in\mathbb{N}$.
\end{lem}
\section{Proof of the Theorems}
\begin{proof}[\bf{Proof of Theorem \ref{Ali-Haldar-thm-2.1}}]
In view of the assumption of the theorem and applying Lemma \ref{lem-1}, we easily obtain that $||A_n||\leq\frac{1-|a_0|^2}{1+\gamma}$ for all $n\in\mathbb{N}$. Using this, it follows that
\beas
M_{T^{f}_{\beta}}(r)&:=& \sum_{n=0}^{\infty}\left( \frac{1}{n+1}\sum_{k=0}^{n}\frac{\Gamma\left( n-k +\beta\right)}{\Gamma\left(n-k+1\right)\Gamma(\beta)} ||A_k||\right)r^n\nonumber\\&=&\sum_{n=0}^{\infty}\left( \frac{1}{n+1}\frac{\Gamma\left( n+\beta\right)}{\Gamma\left(n+1\right)\Gamma(\beta)}||A_0|| \right)r^n\nonumber\\&+&\sum_{n=0}^{\infty}\left( \frac{1}{n+1}\sum_{k=1}^{n}\frac{\Gamma\left( n-k +\beta\right)}{\Gamma\left(n-k+1\right)\Gamma(\beta)} ||A_k||\right)r^n\nonumber\\&\leq&|a_0|\sum_{n=0}^{\infty}\left( \frac{1}{n+1}\frac{\Gamma\left(n+\beta\right)}{\Gamma\left(n+1\right)\Gamma(\beta)} \right)r^n\nonumber\\&+& \frac{1-|a_0|^2}{1+\gamma}\;\sum_{n=0}^{\infty}\left( \frac{1}{n+1}\sum_{k=1}^{n}\frac{\Gamma\left( n-k +\beta\right)}{\Gamma\left(n-k+1\right)\Gamma(\beta)}
\right)r^n.\eeas
From the above inequality, it is easily seen that \bea\label{e4.1a}    M_{T^{f}_{\beta}}(r)&\leq&\frac{|a_0|}{r}\int_{0}^{r}\frac{1}{(1-t)^\beta}dt\;+\frac{(1-|a_0|^2)}{(1+\gamma)}\frac{1}{r}\int_0^{r}\frac{t}{(1-t)^{\beta+1}}\nonumber\\&\leq&\frac{|a_0|}{r}\int_{0}^{r}\frac{1}{(1-t)^\beta}dt\;+\frac{(1-|a_0|^2)}{(1+\gamma)}\frac{1}{r}\left(\int_0^{r}\frac{1}{(1-t)^{\beta+1}}-\int_0^{r}\frac{1}{(1-t)^{\beta}}\right)\nonumber\\&=&\left(\frac{|a_0|}{r}-\frac{1-|a_0|^2}{(1+\gamma)r}\right)\int_0^r\frac{1}{(1-t)^\beta}dt+ \frac{(1-|a_0|^2)}{(1+\gamma)}\frac{1}{r} \int_0^r\frac{1}{(1-t)^{\beta+1}}dt\nonumber\\&=&\left(\frac{|a_0|}{r}-\frac{1-|a_0|^2}{(1+\gamma)r}    \right)\frac{1-(1-r)^{1-\beta}}{1-\beta}+ \frac{(1-|a_0|^2)}{(1+\gamma)}\frac{1}{r} \left(\frac{1-(1-r^{-\beta})}{-\beta}\right)\nonumber\\&=&\frac{1}{r(1+\gamma)}\left(|a_0|(1+\gamma)-(1-|a_0|^2)\right)\frac{1-(1-r)^{1-\beta}}{1-\beta}\nonumber\\&+&(1-|a_0|^2)(\frac{(1-r)^{-\beta}-1}{\beta}):=\Phi(|a_0|,r).\eea
Now differentiating $\Phi(|a_0|,r)$ twice with respect to $|a_0|$, we get
\beas\Phi^{\prime}(|a_0|,r)=\frac{1}{r(1+\gamma)}\left((1+\gamma+2|a_0|)\frac{1-(1-r)^{1-\beta}}{1-\beta}  -2|a_0|\frac{(1-r)^{-\beta}-1}{\beta}\right)\eeas and 
\beas \Phi^{\prime\prime}(|a_0|,r)= \frac{2}{r(1+\gamma)}\left(\frac{1-(1-r)^{1-\beta}}{1-\beta}-\frac{(1-r)^{-\beta}-1}{\beta}\right).\eeas
\par  Let us set \beas F(x):=\frac{1-x^{1-\beta}}{1-\beta}-\frac{x^{-\beta}-1}{\beta},\;\;\text{where}\;\; x=1-r.\eeas Then one can easily see that $F(x)$ is monotonic increasing function of $x\in[0,1)$. This implies that $F(x)\leq F(1)=0$, and thus $\Phi^{\prime\prime}(|a_0|,r)\leq0$, which implies that \bea\label{e4.1}\Phi'(|a_0|,r)\geq  \Phi'(1,r)=\frac{1}{r(1+\gamma)}\left((3+\gamma)\frac{1-(1-r)^{1-\beta}}{1-\beta} -2\frac{(1-r)^{-\beta}-1}{\beta}\right).\eea
\par Consider the function 
\beas G(x):= (3+\gamma)\frac{1-x^{1-\beta}}{1-\beta} -2\frac{x^{-\beta}-1}{\beta},\;\;\text{where}\;\;x=1-r.\eeas
A simple calculation shows that $G^{\prime}(x)\geq0$ for $r\leq(1+\gamma)/(3+\gamma)$, i.e., $G(x)$ is monotonic increasing for $r\leq(1+\gamma)/(3+\gamma)$. Hence $	\Phi'(1,r)$ is monotonic increasing for $r\leq(1+\gamma)/(3+\gamma)$.
\par Again we observe that $\Phi'(1,r)|_{r=0}=0$ and $\lim_{r \to1^-}\Phi'(1,r)\rightarrow-\infty$. Therefore, $\Phi'(1,r)\geq0$ for $r\leq R_{\beta}$, where $R_{\beta}$ is the smallest positive root of the equation \beas (3+\gamma)\frac{1-(1-x)^{1-\beta}}{1-\beta}-2\frac{(1-x)^{-\beta}-1}{\beta}=0.\eeas
\par Therefore, in view of \eqref{e4.1}, it follows that $\Phi'(|a_0|,r)\geq0$ for $0\leq r\leq R_\beta$. This implies that \beas\Phi(|a_0|,r) \leq \Phi(1)= \frac{1}{r}\left(\frac{1-(1-r)^{1-\beta}}{1-\beta}\right).\eeas
Hence in view of \eqref{e4.1a} and the above inequality, we obtain \beas M_{T^{f}_{\beta}}(r) \leq\frac{1}{r}\left(\frac{1-(1-r)^{1-\beta}}{1-\beta}\right)\eeas for $r\leq R_{\beta}$, where $R_{\beta}$ is the smallest positive root of the equation \beas (3+\gamma)\frac{1-(1-x)^{1-\beta}}{1-\beta}-2\frac{(1-x)^{-\beta}-1}{\beta}=0.\eeas
\par In order to show that the radius is sharp, we consider the function
\beas \chi_\gamma(z)=\left(\frac{a-\gamma-(1-\gamma)z}{1-a\gamma-a(1-\gamma)z}\right)I,\eeas where  $z\in\Omega_\gamma$, $a\in(0,1)$, $0\leq\gamma<\frac{1}{2-a}$ and $I$ is the identity operator on $\mathcal{B}(\mathcal{H})$.
\par Let us define the function $L:\Omega_\gamma\to\mathbb{D}$ by $L(z):=(1-\gamma)z+\gamma$ and $\phi:\mathbb{D}\to\mathbb{D}$ by \beas \phi(z):=\frac{a-z}{1-az}\qquad\text{for }a\in(0,1).\eeas
\par It can be easily seen that $\chi_\gamma=\phi\circ L$ maps $\Omega_\gamma$ univalently onto
$\mathbb{D}$. After simple computations, we see that \bea\label{e4.2}
\chi_\gamma(z)=\chi_a(z)I=A_0-\sum_{n=1}^{\infty}A_nz^n,\qquad z\in\mathbb{D},\eea where  $A_0=\left(\frac{a-\gamma}{1-a\gamma}\right)I$
and $A_n=\left(\frac{1-a^2}{a(1-a\gamma)}\left(\frac{a(1-\gamma)}{1-a\gamma}\right)^n\right)I\;\;\text{for }n\in\mathbb{N}.$
\par Now \bea\label{e4.4} M_{T^{f}_{\beta}}(r)&=&\sum_{n=0}^{\infty}\left(\frac{1}{n+1}\sum_{k=0}^{n}	\frac{\Gamma(n-k+\beta)}{\Gamma(n-k+1)\Gamma(\beta)} ||A_k||\right)r^n\nonumber\\&=&\sum_{n=0}^{\infty}\frac{|| A_0||}{n+1}
\frac{\Gamma(n+\beta)}{\Gamma(n+1)\Gamma(\beta)}r^{n}+\sum_{n=0}^{\infty}\left(\frac{1}{n+1}\sum_{k=1}^{n}\frac{\Gamma(n-k+\beta)}{\Gamma(n-k+1)\Gamma(\beta)}||A_k||\right)r^n\nonumber\\&=&\frac{(1-a^2)(1-\gamma)}{(1-a\gamma)^2}\sum_{n=0}^\infty \left(\frac{1}{n+1} \sum_{k=1}^n \frac{\Gamma(n-k+\beta)}{\Gamma(n-k+1) \Gamma(\beta)} \left(\frac{a(1-\gamma)}{1-a\gamma}\right)^n\right)r^n\nonumber\\&+&\frac{a-\gamma}{(1-a\gamma)r}\left(\frac{1-(1-r)^{1-\beta}}{1-\beta}\right).\eea
\par A simple computation shows that \bea\label{e4.5} && \frac{a-\gamma}{(1-a\gamma)r}\left(\frac{1-(1-r)^{1-\beta}}{1-\beta}\right)\nonumber\\&&=\frac{1}{r}\;\frac{1-(1-r)^{1-\beta}}{1-\beta}-\frac{1-a}{r(1-a\gamma)}\left[ (3+\gamma)\frac{1-(1-r)^{1-\beta}}{1-\beta}-2\frac{(1-r)^{-\beta}-1}{\beta}\right]\nonumber\\&&+\frac{2(1-a)}{r(1-a\gamma)}\left[ \frac{1-(1-r)^{1-\beta}}{1-\beta}-\frac{(1-r)^{-\beta}-1}{\beta}\right]\nonumber\\&&=\frac{1}{r}\;\frac{1-(1-r)^{1-\beta}}{1-\beta}-\frac{1-a}{r(1-a\gamma)}\left[ (3+\gamma)\frac{1-(1-r)^{1-\beta}}{1-\beta}-2\frac{(1-r)^{-\beta}-1}{\beta}\right]\nonumber\\&&+\frac{2(1-a)}{r(1-\gamma)}\left[ \frac{1-(1-r)^{1-\beta}}{1-\beta}-\frac{(1-r)^{-\beta}-1}{\beta}\right]+O(1-a)^2.\eea
\par In view of \eqref{e4.4} and \eqref{e4.5}, one can easily observe that $M_{T^{f}_{\beta}}(r)>\frac{1}{r}\;\frac{1-(1-r)^{1-\beta}}{1-\beta}$ whenever  $r>R_\beta$ and $a\rightarrow1^-$. This proves the sharpness.
\end{proof}
\begin{proof}[\bf{Proof of Theorem \ref{t4.4}}]
In view of Lemma \ref{lem-1}, we easily obtain that $||A_n||\leq\frac{1-|a_0|^2}{1+\gamma}$ for all $n\in\mathbb{N}$. Using this, we obtain 
\bea\label{e5.10} M_V^f(r)&=&\sum_{n=0}^{\infty}\frac{||A_n||}{n+1}r^{n+1}=||A_0||r+\sum_{n=1}^{\infty}\frac{||A_n||}{n+1}r^{n+1}\nonumber\\&=& |a_0|r+\frac{1-|a_0|^2}{1+\gamma}\sum_{n=1}^{\infty}\frac{r^{n+1}}{n+1}=|a_0|r-\frac{1-|a_0|^2}{1+\gamma}\;[\ln(1-r)+r]\nonumber\\&:=&\Phi(a_0,r).\eea
\par Differentiating $\Phi(a_0,r)$ twice with respect to $a_0$, we have  
\beas \Phi^{\prime}(a_0,r)&=&r+\frac{2a_0}{1+\gamma}[\ln (1-r)+r],\\\Phi^{\prime\prime}(a_0,r)&=&\frac{2}{1+\gamma}[\ln (1-r)+r]\leq 0\;\;\text{for}\;\;0<r\leq 1.\eeas
\par This implies that $\Phi^{\prime}(a_0,r)\geq\Phi^{\prime}(1,r)=r+\frac{2}{1+\gamma}\;[\ln (1-r)+r]$, which upon simplification gives \beas \Phi^{\prime}(a_0,r)\geq\Phi^{\prime}(1,r) &=&\frac{1}{1+\gamma}\left[(3+\gamma)r+2\ln (1-r)\right]\\&:=&\frac{F(r)}{1+\gamma}.\eeas
\par We observe that $F(0)=0$, $F^{\prime}(r)\geq0$ for $r\leq(1+\gamma)/(3+\gamma)$, $\lim_{r\to 1^-}F(r)=-\infty$, and hence $F(x)=0$ must have a positive root, say $x=R_\gamma$ in $(0,1)$, which depens on $\gamma$. This implies that $\Phi^{\prime}(a_0,r)\geq0$ for $r\leq R_\gamma$. Therefore, in view of \eqref{e5.10}, we obtain \beas M_V^f(r)=\Phi(a_0,r)\leq \Phi(1,r)=r.\eeas
\par It remain to prove the sharpness of the result. For this we consider the same function as defined in \eqref{e4.2}, where $a\in(0,R_\gamma)$.
\par Now \bea\label{e5.11} M_V^{\chi}(r)&=&\sum_{n=0}^{\infty}\frac{||A_n||}{n+1}r^{n+1}=||A_0||r+\sum_{n=1}^{\infty}\frac{||A_n||}{n+1}r^{n+1}\nonumber\\&=& ||A_0||r+\frac{1-||A_0||^2}{1+\gamma}\sum_{n=1}^{\infty}\frac{r^{n+1}}{n+1}\nonumber\\&=&\frac{r(a-r)}{1-ar}+\frac{1-a^2}{a(1-ar)}\sum_{n=1}^{\infty}\left(\frac{a(1-r)}{1-ar}\right)^n\frac{r^{n+1}}{n+1}\nonumber\\&=& r+\frac{r(a-1)(1+\gamma)}{1-a\gamma}+\frac{1-a^2}{a(1-ar)}\sum_{n=1}^{\infty}\left(\frac{a(1-r)}{1-ar}\right)^n\frac{r^{n+1}}{n+1}\nonumber\\&=& r+\frac{(a-1)(1+\gamma)^2}{(1-a\gamma)(3+\gamma)}\left(\frac{r(3+\gamma)}{1+\gamma}+\frac{2\ln (1-r)}{1+\gamma}\right)\nonumber\\&+&\frac{1-a^2}{a(1-ar)}\sum_{n=1}^{\infty}\left(\frac{a(1-r)}{1-ar}\right)^n\frac{r^{n+1}}{n+1}-\frac{2(a-1)(1+\gamma)}{(1-a\gamma)(3+\gamma)}\ln (1-r)\nonumber\\&=& r-\frac{(1-a)(1+\gamma)}{(1-a\gamma)(3+\gamma)}[(3+\gamma)r+2\ln (1-r)]+E_{a,\gamma}(r),\eea where 
\beas && E_{a,\gamma}(r)=\frac{1-a}{1-a\gamma}\left[\frac{1+a}{a}\sum_{n=1}^{\infty}\left(\frac{a(1-\gamma)}{1-a\gamma}\right)^n\frac{r^{n+1}}{n+1}+\frac{2(1+\gamma)}{3+\gamma}\ln (1-r)\right]\\&&= \frac{1-a}{1-a\gamma}\left[\frac{(1+a)(1-a\gamma)}{a^2(1-\gamma)}\left(-\ln(1-\frac{r(1-a\gamma)}{a(1-\gamma)})-\frac{r(1-a\gamma)}{a(1-\gamma)}\right)+\frac{2(1+\gamma)\ln (1-r)}{3+\gamma}\right].\eeas
\par Now after simple computations it is easy to observe that $E_{a,\gamma}(0)=0$ and $E_{a,\gamma}^{\prime}(r)> 0$ for $r>R_\gamma$. Therefore, $E_{a,\gamma}>0$ for $r>R_\gamma$. 
\par Hence in view of \eqref{e5.11} and the fact that $(3+\gamma)r+2\ln (1-r)<0$ for $r>R_\gamma$, it is easily seen that $M_V^\chi(r)>r$. This shows the sharpness. 
\end{proof}
\noindent\textbf{Conflict of interest:} The authors declare that there is no conflict  of interest regarding the publication of this paper.\vspace{1.2mm}

\noindent {\bf Funding:} Not Applicable.\vspace{1.2mm}

\noindent\textbf{Data availability statement:}  Data sharing not applicable to this article as no datasets were generated or analysed during the current study.\vspace{1.2mm}

\noindent {\bf Authors' contributions:} All the authors have equal contributions in preparation of the manuscript.

\end{document}